\documentclass[a4paper,12pt,draft]{amsart}
\usepackage[margin=30mm]{geometry}
\usepackage{amssymb, amsmath, amsthm, amscd}

\usepackage[T1]{fontenc}
\usepackage{eucal,mathrsfs,dsfont}
\usepackage{color}

\definecolor{mno}{rgb}{0.5,0.1,0.5}

\newcommand{\R}{\mathds R}
\newcommand{\N}{\mathds N}

\newcommand{\I}{\mathds 1}
\newcommand{\Z}{\mathds Z}

\newcommand{\lip}{\textmd{Lip}}
\newcommand{\dom}{{\mathcal F}}
\newcommand{\form}{{\mathcal E}}
\newcommand{\cE}{\mathcal{E}}
\newcommand{\cEc}{\mathcal{E}^{(c)}}
\newcommand{\cEj}{\mathcal{E}^{(j)}}
\newcommand{\cF}{\mathcal{F}}
\newcommand{\Ga}{\Gamma}

\newcommand{\al}{{\alpha}}
\newcommand{\be}{{\beta}}
\renewcommand{\th}{{\th}}

\newcommand{\ep}{{\epsilon}}

\newtheorem{theorem}{Theorem}[section]
\newtheorem{lemma}[theorem]{Lemma}
\newtheorem{proposition}[theorem]{Proposition}

\theoremstyle{definition}

\newtheorem{example}[theorem]{Example}
\newtheorem{remark}[theorem]{Remark}

\begin{document}
\allowdisplaybreaks
\title[Conservativeness and recurrence of jump-diffusions]
{On the conservativeness and the recurrence
of symmetric jump-diffusions}
\author{Jun Masamune}
\address{Department of Mathematics and Statistics,
Penn State, Altoona, 3000 Ivyside Park, Altoona, Pennsylvania 16601 US}
\email{jum35@psu.edu}

\author{Toshihiro Uemura}
\address{Department of Mathematics, Faculty of Engineering Science,
Kansai University, Suita-shi, Osaka 564-8680, Japan}
\email{t-uemura@kansai-u.ac.jp}

\author{Jian Wang}
\address{School of Mathematics and Computer Science, Fujian Normal
University, 350007, Fuzhou, P.R. China}
\email{jianwang@fjnu.edu.cn}


\date{\today }

\maketitle{}
\frenchspacing

\begin{abstract}
Sufficient conditions for a symmetric jump-diffusion process to be
conservative and recurrent are given in terms of the
volume of the state space and the jump kernel of the process.
A number of examples are presented to illustrate the optimality of these conditions;
in particular, the situation is allowed to be that the state space
is topologically disconnected but the particles can jump from a
connected component to the other components.
\end{abstract}

\tableofcontents

\section{Introduction and Main Results}\label{section1}
Let $(X,d,m)$ be a metric measure space. We assume that every metric ball
$B(x,r)=\{z\in X: d(x,z)<r\}$ centered at
$x \in X$ with radius $r>0$ is pre-compact, and
the measure $m$ is a Radon measure with full support.
In particular, $X$ is locally compact and separable.
Let $(\form,\dom)$ be a regular symmetric Dirichlet form in
$L^2(X;m)$.
We denote the extended Dirichlet space
of $(\cE,\cF)$ by $\cF_e$, and a quasi continuous version of $u \in \cF_e$ by $\tilde{u}$.
According to the Beurling-Deny theorem,
see, e.g., \cite[Theorem 3.2.1 and Lemma 4.5.4]{FOT94}, we can express $(\cE,\cF)$ as follows
 \begin{align*}
\form(u,v)= &\,\, \form^{(c)}(u,v)+\iint_{x\not=y}\bigl(\tilde{u}(x)-\tilde{u}(y)\bigr)
\bigl(\tilde{v}(x)-\tilde{v}(y)\bigr)\,J(dx,dy) \\
 &  + \int_X \tilde{u}(x)\tilde{v}(x)\,k(dx)
    \quad \mbox{ for any $u,v \in \dom_e$},
\end{align*}
where
$(\form^{(c)}, C_0(X)\cap \dom)$  is a strongly-local symmetric form
and $C_0(X)$ is the space of all real-valued  continuous functions on $X$
with compact support;
    $J$ is a symmetric positive Radon measure
    on the product space $X\times X$ off the diagonal
$\{(x,x): x\in X\}$; and
    $k$ is a positive Radon measure on $X$.

\medskip

Let $\mu_{\langle \cdot, \cdot \rangle}$ be a bounded
signed measure, see \cite[Lemma 3.2.3]{FOT94}, such that
\[
\form^{(c)}(u,v)=\frac 12 \mu_{\langle u,v\rangle}(X)=\frac 12 \int_X
\mu_{\langle u,v\rangle}(dx) \quad\textrm{ for } u,v \in \dom_e.\]
Throughout the paper, we assume the following set (A) of conditions:
\begin{itemize}
\item[{ (A-1)}] the killing measure $k$ does not appear;
    that is, the corresponding
process is {\it no killing inside}.
\item[{ (A-2)}] for each $u,v\in \dom_e$,
the measure $\mu_{\langle u,v\rangle}$ is absolutely continuous with
respect to $m$.
We denote the corresponding Radon-Nikodym density by $\Gamma^{(c)}(u,v)$;
namely,
$$\mu_{\langle u,v\rangle}(dx)=\Gamma^{(c)}(u,v)(x) \,m(dx).$$
\item[{ (A-3)}] the jump measure $J$ has a symmetric kernel
$j(x,dy)$ over $X\times {\mathcal B}(X)$ such that
$$
J(dx,dy)=j(x,dy)\,m(dx) \ \Bigl(=j(y,dx)\,m(dy)=J(dy,dx) \Bigr).
$$
\end{itemize}
For $u,v \in \dom_e$, define
\[\Gamma^{(j)}(u,v)(x)= \int_{x\not=y}
\bigl(\tilde{u}(x)-\tilde{u}(y)\bigr)
 \bigl(\tilde{v}(x)-\tilde{v}(y)\bigr)\,j(x,dy),\]
and
\[\form^{(j)}(u,v)  =  \int \Gamma^{(j)}(u,v)(x)\,m(dx).\]
Therefore, the form $\form$ has the following expression for any $u,v \in \dom_e$:
\begin{align*}
\form(u,v)\! = &\form^{(c)}(u,v)+\form^{(j)}(u,v) \\
 =&\frac 12 \int_X \Gamma^{(c)}(u,v)(x)\,m(dx) +
 \int_X\Gamma^{(j)}(u,v)(x)\, m(dx) \\
= &\frac 12 \int_X \Gamma^{(c)}(u,v)(x)\,m(dx)\!\! +\!\! \iint_{x\not=y}\!
\bigl(\tilde{u}(x)-\tilde{u}(y)\bigr)
 \bigl(\tilde{v}(x)-\tilde{v}(y)\bigr)\,j(x,dy)\,m(dx).
\end{align*}

Let $\psi_K$ be the distance function from a compact set $K$ of $X$, i.e., $\psi_K (\cdot) = \inf_{y \in K} d( \cdot ,y)$.
For every $r>0$, we denote $B(K,r)=\{ x \in X: \psi_K <r\}$ and its closure $\{ x \in X: \psi_K \le r\}$ by $\overline{B}(K,r)$.
Clearly, $B(K,r)$ is pre-compact.
Let $\dom_{\rm loc}$ be the set of measurable functions $u$
such that for each relatively compact open set $G$ of $X$
there exists $w\in \dom$ which satisfies that $u|_G=w|_G$ $m$-a.e.
Additionally, we assume the following set (M) of conditions
so that both $\form^{(c)}$ and $\form^{(j)}$ are compatible with the distance $d$:

\medskip

\begin{itemize}
\item[{ (M-1)}]\ $\psi_K \in \dom_{\rm loc}$ for every compact set $K \subset X$,

\medskip

\item[{ (M-2)}]\  $M_c :={{\rm{ess\, sup}}_{x\in X^{(c)}}} \ \Ga^{(c)}(d,d) (x) < \infty$,

\medskip

\item[{ (M-3)}]\ $M_j:= {{\rm{ess\, sup}}_{x \in X^{(j)}}} \int_{x\not=y} \bigl(1 \wedge d^2(x,y) \bigr) \,  j(x,dy) < \infty,$
\end{itemize}
where $X^{(c)} = \{ x \in X: \Gamma^{(c)} \neq 0\}$ and $X^{(j)} = \{ x \in X: \Gamma^{(j)} \neq 0\}$.

\medskip

There are many classical examples of symmetric diffusions or
symmetric pure jump processes whose Dirichlet form satisfies
conditions (A) and (M): for instance, strongly-local Dirichlet forms
on a metric measure space, whose distance is the
Carnot-Carath\'eodori distance associated with the Dirichlet form.
This includes canonical Dirichlet forms on Riemannian manifolds, CR
manifolds, sub-Riemannian manifolds, and weighted manifolds;
divergence type operators with bounded coefficients on Euclidean
spaces; the sum of squares of vector fields satisfying H\"ormader's
condition, the quantum graphs, and pre-fractals.
Other examples are symmetric $\alpha$-stable L\'evy processes with
$\alpha \in (0,2)$ on Euclidean spaces,  and symmetric random walks
on graphs.

\medskip

Let $A$ be the generator of $(\form,\dom)$ in $L^2(X;m)$.
We denote the associated semigroup and the resolvent by
$(T_t)_{t\ge0}=(e^{tA})_{t\ge0}$ and $G=\int^\infty_0 T_t\, dt$, respectively.
The Dirichlet form $(\cE,\cF)$ is called \emph{conservative} if
\[ T_t 1 \equiv 1,\quad \mbox{ $m$-a.e.\ for any $t>0$} \]
and \emph{recurrent} if
\[G f (x) \equiv 0 \mbox{ or } \infty \mbox{  for any  $f \in L^1_+ (X; m)$
and $m$-a.e.\ $x \in X$}.\] It is a classical result that Brownian
motion on $\R^n$ is conservative for any $n \ge 1$ and is recurrent
if and only if $n=1,2$. This result has been generalized to the
Wiener process of complete Riemannian manifolds, and
one of
the most important discoveries is that a certain bound of the volume
at infinity -- rather than the dimension -- implies these
properties.
This fact was
first found by M.P.\ Gaffney \cite{Gaffney.59} for the
conservativeness, and it has been refined by various methods in
\cite{Azencott.74, KarpLi.83, Takeda.89, Hsu.89,
Davies.92,Grigor'yan.99}. Especially, R.\ Azencott \cite{Azencott.74}
and A.\ Grigor\rq{}yan \cite{Grigor'yan.99} demonstrated that the
conservativeness may fail without a condition on the curvature or
volume. On the other hand, the recurrence of the Wiener process of
Riemannian manifolds or jump processes has been investigated by
several authors in
\cite{ChengYau.75, Karp.82, Varopoulos.83, Grigoryan.83,
Grigoryan.85, Okura.96}. Furthermore, K.\ Th.\ Sturm \cite{Sturm.94}
extended the theory to a general strongly-local regular Dirichlet
form on a metric measure space equipped with the
Carnot-Carath\'eodori distance.

Recently, there has been a tremendous amount of work devoted
to the conservation property of a non-local Dirichelt form; for instance,
the physical Laplacian on an infinite graph
\cite{DodziukMathai.06, Dodziuk.06, Weber.10, Wojciechowski.09,
Wojciechowski.11, KellerLenz.11, Huang.11a, Huang.11b, Huang.12}
and non-local Dirichlet forms
\cite{MasamuneUemura.11, Grigor'yanHuangMasamune.10, Shiozawa.11};
however, as far as the authors know, there is only
one result by Z.-Q.\ Chen and T. Kumagai \cite{ChenKumagai.10} for the Dirichlet form which has both the strongly-local and non-local terms.
Due to its nature, the associated process is called a \emph{jump-diffusion process}.

\bigskip

Our first main purpose is to investigate the conservative property of a jump-diffusion process.
For any $x \in X$ and $r>0$,  the volume of $\overline{B}(x,r)$ is denoted by $V(x,r)$.

\begin{theorem} \label{theorem1}
If
\begin{equation} \label{eq;theorem1.1}
\liminf_{r \to \infty}
    \frac{\ln V(x_0,r)}{r \ln r} < \infty,
\end{equation}
for some $x_0 \in X$, then $(\cE,\cF)$ is conservative.
\end{theorem}

This result was obtained for a non-local Dirichlet form in
\cite[Theorem 1.1]{Grigor'yanHuangMasamune.10}, where the left-hand side of \eqref{eq;theorem1.1} is required to be less than  $1/2$.
Let us explain the significance of removing the constant 1/2
by comparing the uniqueness class with
the conservation property.
Let $\mathcal U$ be the set of the solutions to the Cauchy problem of the
heat equation with zero initial data. 
If any $u \in \mathcal U$ is identically 0, then $\mathcal U$ is called a
\emph{uniqueness class}.
Under an integrability assumption,
determining the uniqueness class implies
the conservativeness of Riemannian manifolds \cite{Grigoryan.86},  Dirichlet forms \cite{Sturm.94}, and graphs \cite{Huang.12}.
In fact, A.\ Grigor'yan \cite{Grigoryan.86} and K.\ Th.\ Sturm \cite{Sturm.94}
established the sharp conservation test for complete Riemannian manifolds
and strongly-local Dirichlet forms, respectively, in this way.
However, X.\ Huang \cite[Section 3.3]{Huang.12} constructed an example of a graph,
which verifies that the constant 1/2 is indeed needed
for the uniqueness class.
Therefore, Theorem \ref{theorem1} together with Huang's example
demonstrates that the uniqueness class condition is
really stronger than the conservation property for a graph.

\bigskip

Next, we turn to the recurrence.
For any $x\in X$ and $r>0$, the volumes of  the closed ball $\overline{B}(x,r)$ intersected with $X^{(c)}$ and $X^{(j)}$ are denoted by $V^{(c)}(x,r)$ and $V^{(j)}(x,r)$, respectively.
For $r>0$, define
\[\omega (r) = \sup_{x \in {X^{(j)}}} \int_{x \neq y}
    \left( d(x,y) \wedge r \right)^2\, j(x,dy).
\]
 Our second main result is
\begin{theorem} \label{theorem2}
If
\begin{equation}\label{thm21}
\liminf_{r \to \infty} \frac{1}{r^2}\bigg[
     V^{(c)}(x_0,r)+V^{(j)}(x_0,r) \omega (r) \bigg]<\infty,\end{equation}
for some $x_0 \in X$, then $(\cE,\cF)$ is recurrent.
\end{theorem}

Theorem \ref{theorem2} was proven in the case of the Wiener process (namely, the process does not jump) on a complete Riemannian manifold by S.Y.\ Cheng and S.T.\ Yau \cite{ChengYau.75}. Theorem \ref{theorem2} is sharp for an isotropic symmetric $\alpha$-stable L\'{e}vy process on $\R^n$, see, e.g., \cite[Corollary 37.17 and Theorem 37.18]{SA} or Example \ref{th23} in Section \ref{examples}. Here, let us mention that \cite[Corollary 37.17 and Theorem 37.18]{SA} are derived from the characteristic functions of the associated processes, see \cite{SW} for the recent development on this topic; while Theorem \ref{theorem2} is based on the theory of Dirichlet forms.

\bigskip

This paper is organized as follows.
Section \ref{preliminaries} is devoted to the preliminaries.
Here we establish an integral-derivation type property for a Dirichlet form of jump-process type, which is a technical key to prove the conservation property.
The main results, Theorems \ref{theorem1} and \ref{theorem2}, are proved
in Sections \ref{proofoftheorem1} and \ref{proofoftheorem2},
respectively.
Finally, in Section \ref{examples} we present some examples
of symmetric jump-diffusions to illustrate the power of our main theorems.


\section{Preliminaries: the Integral-Derivation Property}
\label{preliminaries}
In this section, we first prepare the preliminaries
and then proceed to establish an integral-derivation type property
for a Dirichlet form with jump-diffusion type.
This will be used to prove the
conservation property in the next section.

We begin with the following quite elementary fact.
\begin{lemma} \label{lemma-1}
If $u \in \dom_{\rm loc} \cap L^\infty$
has compact support, where $L^\infty = L^\infty(X)$ is the space of
real-valued bounded measurable functions on $X$, then
$u \in \dom\cap L^\infty$.
\end{lemma}
\begin{proof}
Suppose that supp\,$u\subset K$ with a compact set $K$.
Let $\eta \in \dom \cap L^\infty$ agree with $u$ on $B(K,1)$.
Because of the regularity and the fact that the constant function belongs to $\dom_{\rm loc}$,
see the remark in \cite[Page 117]{FOT94}, there is a function $\chi \in \dom \cap L^\infty$
such that $\chi|_{K} = 1$ and supp\,$\chi \subset B(K,1)$.
Since $\eta\chi \in \dom$ and $u =\eta\chi$, the statement follows.
\end{proof}

For the sake of simplicity, hereafter we denote $\Gamma[\cdot]=\Gamma(\cdot,\cdot)$,
$\cE [\cdot]= \cE(\cdot,\cdot)$, etc.
We say that the jump range of $\form$ or $\form^{(j)}$
is uniformly bounded, if there exists a constant $a>0$ such that
supp$(j(x,\cdot)) \subset B(x,a)$ for every $x \in X$.

\begin{lemma} \label{lemma0}
Suppose that the jump range of $\form$ is uniformly bounded.
If $u \in \dom_{\rm loc} \cap L^\infty$ is constant outside a
compact set, then for any $v\in \dom\cap L^\infty$, $uv\in \dom \cap L^\infty.$
\end{lemma}
\begin{proof} Let $K \subset X$ be a compact set such that $u$ is constant outside it.
Consider the sequence of cut-off functions $(\chi_l)_{l\in\N}$, where for $l \ge 1$,
\[\chi_l = \left( (2 - l^{-1}\psi) \wedge 1 \right)_+.\]
By Lemma \ref{lemma-1}, the function $\chi_l $ belongs to $\dom$ for any $l\ge1$. Obviously, $\chi_l = 1$ on $B(K,l)$ and supp$(\chi_l)
    \subset 
\overline{B}(K,2l)$.

    We set for any $l \ge 1$, $v_l =u v \chi_l$.
Since $u  \in \dom_{\rm loc} \cap L^\infty$ and $v\in \dom\cap L^\infty$,
$v_l$ belongs to $\dom_{\rm loc} \cap L^\infty$ and has compact support.
Hence, Lemma \ref{lemma-1} shows that $v_l \in \dom$ for any $l\ge1$.

Next, we claim that the sequence $(v_l)_{l \ge 1}$ is $\form$-Cauchy.
Set $\chi_{l,l\rq{}} = \chi_l - \chi_{l\rq{}}$ for $l,l\rq{} \ge 1$.
Since the jump range of $\form$ is uniformly bounded,
for large enough $l$ and $l\rq{}$,
\begin{align*}
\form [ v_l - v_{l\rq{}} ]
=
\form [ (\chi_l - \chi_{l\rq{}}) u v]
= \kappa\cdot \form [ \chi_{l,l\rq{}}v  ],
\end{align*}
where $\kappa=u|_{K^c}$.
By \cite[Lemma 3.2.5]{FOT94},
\[
\form^{(c)} [ \chi_{l,l\rq{}} v ]
\le
 2 \int v^2 \Gamma^{(c)} [ \chi_{l,l\rq{}}]\,dm + 2 \int  \chi^2_{l,l\rq{}} \Gamma^{(c)} [v]\,dm.\]
Because of (M) and the chain rule of the strongly-local Dirichlet form, see, e.g., \cite[Page 190]{Sturm.94},
$ \Gamma^{(c)} [ \chi_{l,l\rq{}}]\to0$ as $l,l\rq{}\to\infty$.
This together with the fact $\chi_{l,l\rq{}}\to0$ as $l,l\rq{}\to\infty$ yields that
$\form^{(c)} [ \chi_{l,l\rq{}} v ]$ tends to zero as $l,l\rq{}\to\infty$.

On the other hand,
\begin{align*}
\form^{(j)} [ \chi_{l,l\rq{}} v]
\le
&\,2 \int v^2(x) \int ( \chi_{l,l\rq{}} (x) - \chi_{l,l\rq{}} (y) )^2\, j(x,dy)\,m(dx) \\
&\,+
2\iint \chi_{l,l\rq{}} ^2 (y) (v(x) - v(y))^2\, j(x,dy)\,m(dx)\\
=&: { (I) + (I\!I)}.
\end{align*}
For any $x\in X$,  \begin{align*}
   &\int ( \chi_{l,l\rq{}} (x) - \chi_{l,l\rq{}} (y) )^2\, j(x,dy) \\
&=
\int  \big( (\chi_l (x) - \chi_l (y)) -  ( \chi_{l\rq{}} (x) - \chi_{l\rq{}} (y) ) \big)^2\, j(x,dy) \\
&\le
2
\int ( \chi_l (x) - \chi_l (y) )^2\, j(x,dy)
+
2\int ( \chi_{l\rq{}} (x) - \chi_{l\rq{}} (y) )^2\, j(x,dy) \\
&\le
2 ( l^{-2}+{l\rq{}}^{-2}) \int d(x,y)^2\, j(x,dy).
\end{align*}
Combining the fact that supp\,$(j(x,dy)) \subset B(x,a)$ for all $x\in X$ and some $a>0$ with the assumption (M),
the last term in the right-hand side of the equation above is dominated by
\begin{align*}
2(1 + a^2)M_j \left(l^{-2} + {l\rq{}}^{-2} \right),
\end{align*}
which tends to 0 as $l,l\rq{}\to\infty$.
Hence ${(I)}\to 0$ as $l,l\rq{}\to\infty$. Since $ \chi_{l,l\rq{}} \to 0$, $m$-a.e.\ as $l,l\rq{}\to\infty$,
${(I\!I)} \to 0$ as $l,l\rq{}\to\infty$.
Thus, $\form^{(j)} [ \chi_{l,l\rq{}} v]\to 0$ as $l,l\rq{}\to\infty$,
and so the desired claim follows.

Finally, since $v_l \to uv $, $m$-a.e.\ as $l\to\infty$, $ uv \in \dom_e$.
This together with the fact $ u v\in L^2$ and \cite[Theorem 1.5.2 (iii)]{FOT94} yields that $u v \in \dom$.
\end{proof}


The following is the integral-derivation property for our Dirichlet form.
\begin{lemma} \label{lemma1}
Suppose that the jump range of $\form$ is uniformly bounded.
If $u \in \dom \cap L^\infty$ and $\phi \in \dom_{\rm loc} \cap L^\infty$ is constant outside a
compact set, then
\begin{equation} \label{eq; 2.5}
\form(u, u \phi)    =\int  u \Gamma (u,\phi)\,dm
    +\int \phi \Gamma[u]\,dm,
\end{equation} where $\Gamma=\frac{1}{2}\left(\Gamma^{(c)}+\Gamma^{(j)} \right).$
 \end{lemma}

\begin{proof}
According to Lemma \ref{lemma0}, $u \phi \in \dom$.
By the derivation property of $\form^{(c)}$,
see, e.g., \cite[Lemma 3.2.5 and the note in Page 117]{FOT94},
\[
\int\Gamma^{(c)} (u,  u \phi )\,dm
    =\int  u \Gamma^{(c)} (u,\phi)\,dm
    +\int \phi \Gamma^{(c)} [u]\,dm.
\]
%
Next,
by the integral property of
a non-local Dirichlet form, see \cite[Proposition 2.2]{MasamuneUemuraPreprint},
we have
\[ \int\Gamma^{(j)} (u,  u \phi )\,dm
    =\int u \Gamma^{(j)} (u,\phi)\,dm
    +\int \phi \Gamma^{(j)} [u]\,dm.\]
Combining the two identities, we obtain (\ref{eq; 2.5}).
  \end{proof}

\section{Proof of Theorem \ref{theorem1}: the Conservation Property}
\label{proofoftheorem1}
The aim of this section is to prove Theorem \ref{theorem1}.
For any $a>0$, consider a symmetric form $({\form^{(j,a)}}, \dom)$ defined by
\[ {\form^{(j,a)}} [u] = \iint (u(x)-u(y))^2 \I_{\{d(x,y)\le a\}}\, j (x,dy)\,m(dx) \quad
\mbox{ for $u \in \dom$}. \]
Under the condition (M),
$({\form^{(j,a)}}+ \form^{(c)},\dom)$ is a regular Dirichlet form,
and it is conservative if and only if so is $(\form,\dom)$,
see \cite[Section 4]{RU} and \cite[Section 3]{MasamuneUemura.11}.
Clearly,   $({\form^{(j,a)}}, \dom)$ has uniformly bounded range.
Therefore, in order to prove the conservation property,
we may and do assume that $\form$ has uniformly bounded jump range.
More precisely, we suppose that there exists a constant $a>0$ such that
\[ j(x,dy) = \I_{B(x,a)}(y)\, j (x,dy)\quad
\mbox{ for all $x\in X$}.\]
Our proof is basically the Davies method \cite{Davies.92}, which was used also in \cite{Grigor'yanHuangMasamune.10}; however, we are able to get a better result because of the choice of $a$. In this section, the constant $a$ will be
\begin{equation}\label{constant;c}
a=a(x_0,m):=\bigg[
    8\,\liminf_{r\rightarrow\infty} \frac{\log V \left( x_0, r\right)
}{r\log r}+9\bigg]^{-1},
\end{equation}
where $x_0 \in X$ is the reference point in Theorem \ref{theorem1}.
For $f\in C_0(X)$ with $f \ge 0$,
set
\[
\psi(x)=d(x, \textrm{supp} (f))
\] and
\[
\phi(x)=e^{\alpha\psi(x)},
\]
where $\alpha>0$ is a constant determined  later.
Note that if $n \ge 1$ and $x \in X$ satisfy
\[n\ge a^{-1}\big[4a+2d\,(x_0,\,\textrm{supp} (f))\big]
\quad\mathrm{ and }\quad
(n-2)a\le d(x,x_0)\le (n+1)a,\]
 then
\[
\psi(x)\ge d(x,x_0)- d(x_0,\,\textrm{supp} (f)
)
\ge (n-2)a- d(x_0,\,\textrm{supp} (f)
)
\ge {an}/{2},
\]
and so
\begin{equation}\label{proof111}
    \phi(x)=e^{\alpha\psi(x)}\ge e^{{a\alpha n}/{2}}.
\end{equation}
For  the function $f$ above and any $t\ge0$,  we denote $u_t = T_t f$.
Since $(T_t)_{t\ge0}$ is analytic, $u_t$ belongs to the
domain of the $L^2$-generator $A$ of $(\form,\dom)$;
in particular, $u_t \in \dom \cap L^\infty$ for any $t>0$.

\medskip

The following lemma provides the key estimate.
\begin{lemma} \label{lemma2}
Using the notations above, for any $t\ge0$,
\begin{equation}\label{proof6th1}
\int^t_0 \int\phi\,\Gamma [u_s] \,dm\,ds
\le 2 e^{\gamma t} \|\phi^{1/2}f\|_2^2 ,
\end{equation}
where $\gamma = {{\alpha^2 ( e^{2\alpha a}+1) }M /{2}}$ and $M=M_c \vee M_j$.
\end{lemma}
\begin{proof}
In the following, we denote  the norm and the inner product of $L^2(X;m)$ by $\|\cdot\|_2$ and $\langle \cdot, \cdot\rangle$,
respectively.
For any $n\ge1$, set
\[\phi_n(x)=e^{\alpha (\psi(x) \wedge n)}.\]
Since $\psi\in \dom_{\rm loc}$, we may apply
an argument in \cite[Pages 116 -- 117]{FOT94} to deduce that
$\phi_n \in \dom_{\rm loc}$ for every $n\ge1$.
Taking into account that $\psi \in L^\infty$ is constant outside a compact set,
Lemma \ref{lemma0} shows that for every $t>0$ and $n\ge1$,
$u_t \phi_n \in \dom$.
Therefore, by Lemma \ref{lemma1}, for all $t>0$,
\begin{align*}
\frac{1}{2} \frac{d}{dt} \| \phi^{1/2}_n u_t \|^2_2
&=\langle\dot{u}_t, \phi_nu_t \rangle \\
&=-\cE (u_t, \phi_n u_t) \\
&=-\int \phi_n\,\Gamma [u_t] \,dm
    -\int u_t \,\Gamma (u_t,\phi_n)\,dm\\
&\le -\int \phi_n\,\Gamma [u_t] \,dm
    +\Bigg| \int u_t \,\Gamma (u_t,\phi_n)\,dm\Bigg|,
\end{align*}
where $\dot{u}_t=\frac{d}{dt}u_t$.
This is,
\begin{equation}\label{proof3th1}
\int \phi_n\,\Gamma [u_t] \, dm \le \Bigg| \int u_t \,\Gamma (u_t,\phi_n)\,dm\Bigg| -
\frac{1}{2} \frac{d}{dt} \| \phi^{1/2}_n u_t \|^2_2.
\end{equation}
Next, we estimate the first term on the right side of this equation.
For every $x \in X$, according to the Cauchy-Schwartz inequality,
\[\aligned
&|\Ga^{(j)} (u_t , \phi_n) (x)| \\
&=\bigg|  \int \left( u_t(x) - u_t(y)\right)  \left( \phi_n(x) - \phi_n(y)\right)\, j (x,dy) \bigg| \\
&\le \sqrt{ \int \left( u_t(x) - u_t(y)\right)^2\, j (x,dy)   }
     \sqrt{ \int \left( \phi_n(x) - \phi_n(y)\right)^2\, j (x,dy)   } \\
&=\sqrt{\Ga^{(j)} [u_t] (x)} \sqrt{\Ga^{(j)}[\phi_n] (x)}.
\endaligned\]
By the Cauchy-Schwartz inequality again,
\[
\aligned
\Bigg| \int u_t \,\Gamma^{(j)} (u_t,\phi_n)\,dm \Bigg|
\le & \int \phi^{1/2}_n \,\sqrt{\Ga^{(j)} [u_t]} \,\,\phi^{-1/2}_n \sqrt{ u_t^2 \,\Ga^{(j)} [\phi_n]}\, dm  \\
\le &\sqrt{ \int \phi_n\,\Gamma^{(j)} [u_t] \, dm}\, \sqrt{\int\phi^{-1}_n \,u_t^2\,
    \Gamma^{(j)} [\phi_n]\, dm}.
\endaligned\]
Since
\[
|e^{\alpha r} - 1 | \le \alpha \,e^{\alpha a}|r| \quad
\mbox{for any $r \in (0,a]$},
\]
it follows that
$$|\phi_n(x)-\phi_n(y)|\le \alpha e^{\alpha a}\,\phi_n(x)\,d(x,y)\quad
\mbox{ for any $x, y\in X$ with
$d(x,y)\le a$},
$$
and so
\[
\Gamma^{(j)}[ \phi_n ] (x) \le \big(\alpha e^{\alpha a} \phi(x)\big )^2 \int d^2(x,y)\, j (x,dy)
\quad \mbox{for every $x \in X$.}
\]
Since supp$(j(x,dy)) \subset B(x,a)$
for any $x\in X$ and some constant $a\in(0,1)$, we get
$$\aligned &\int\phi^{-1}_n \,u^2_t\,
    \Gamma^{(j)} [\phi_n] \, dm \\
&\le\alpha^2 e^{2\alpha a}\int\phi_n(x)\,u^2_t (x) \int d(x,y)^2\,j(x,dy)\,m(dx)\\
&\le\alpha^2 e^{2 \alpha a}\int\phi_n(x)\,u^2 _t(x)
    \int \left(d(x,y)\wedge a\right)^2\,j (x,dy)\,m(dx)\\
&\le M_j \alpha^2 e^{2\alpha a}
    \int\phi_n \,u^2_t \,dm.\endaligned$$
Therefore, for any
$\lambda>0$,
\begin{equation*}
\aligned &
\Bigg| \int u_t\,\Gamma^{(j)} (u_t,\phi_n)\,dm \Bigg| \\
& \le\sqrt{ M_j \int \phi_n \,\Gamma^{(j)} [u_t] \,dm}
    \sqrt{\alpha^2 e^{2 \alpha a} \int\phi_n \,u_t^2 \,dm}\\
 &\le \frac{M_j}{2\lambda}\int \phi_n\,\Gamma^{(j)} [u_t]\,dm
    +\frac{\lambda \alpha^2 e^{2 \alpha a}}{2} \int\phi_n \,u_t^2 \,dm\\
 &= \frac{M_j} {2\lambda}\int \phi_n \,\Gamma^{(j)} [u_t] \, dm
    +\frac{\lambda \alpha^2 e^{2 \alpha a}}{2}
    \|\phi_n^{1/2} u_t \|_2^2, \endaligned
\end{equation*}
where in the last inequality we have used the fact that
$2\xi\eta\le \lambda^{-1}\xi^2+\lambda \eta^2$ for any $\xi, \eta\ge0$ and $\lambda>0.$

On the other hand, we apply the argument above for the local term to get that
\begin{equation*}
\aligned
\Bigg| \int u_t\, \Gamma^{(c)} (u_t,\phi_n)\,dm \Bigg|
\le \sqrt{ \int \phi_n\,\Gamma^{(c)} [u_t] \, dm}\,\sqrt{\int\phi^{-1}_n \,u_t^2\,
    \Gamma^{(c)}[\phi_n]\, dm}.\endaligned
\end{equation*}
According to the chain rule for a strongly-local Dirichlet form, see, e.g., \cite[Page 190]{Sturm.94},
$$\int\phi^{-1}_n \,u^2_t\,
    \Gamma^{(c)} [\phi_n] \, dm \le \alpha^2 \int u^2_t\, \phi_n\,
    \Gamma^{(c)} [d] \, dm ,$$
which along with the assumption (M) gives us
$$\int\phi^{-1}_n \,u^2_t\,
    \Gamma^{(c)} [\phi_n] \, dm \le M_c\alpha^2 \int u^2_t\, \phi_n \, dm.$$
We again follow the argument above to obtain the estimate:
$$\Bigg| \int u_t\, \Gamma^{(c)} (u_t,\phi_n)\,dm \Bigg|
\le \\
 \frac{M_c}{ 2 \lambda } \int \phi_n \,\Gamma^{(c)} [u_t] \,dm +
\frac{\lambda \alpha^2  }{2} \|\phi^{1/2}_n u_t \|^2_2 \quad
\mbox{ for any $\lambda>0$}.$$
Combining the estimates for the non-local and strongly-local terms, we get that
\[
\Bigg| \int u_t \,\Ga (u_t,\phi_n)\,dm \Bigg|
\le
\frac{M} {2\lambda}\int \phi_n \,\Ga [u_t] \, dm
    +\frac{\lambda  \alpha^2 (e^{2 \alpha a}+1)}{2}
    \|\phi_n^{1/2} u_t \|_2^2.\]
By applying this inequality for \eqref{proof3th1}, we have
\begin{equation}\label{proof4th1}
 \bigg( 2-\frac{M }{\lambda}\bigg)
    \int \phi_n\,\Ga [u_s]\, dm
\le{\lambda    \alpha^2  ( e^{2\alpha a} +1)}\,\|\phi_n^{1/2}  u_s
    \|_2^2 - \frac{d}{ds}\|\phi_n^{1/2} u_s\|_2^2.
\end{equation}
If we integrate this with respect to $s$ over $[0,t]$, then
\begin{equation}\label{proof5th1}
\aligned
&\left( 2-\frac{M }{\lambda}\right)
    \int_0^t\int \phi_n\,\Ga [u_s]\, dm \\
&\le  \lambda  \alpha^2  (e^{2\alpha a}+1)
    \int_0^t\|\phi_n^{1/2} u_s\|_2^2\,ds - \left( \|\phi_n^{1/2}u_t\|_2^2-\|\phi_n^{1/2}f\|_2^2 \right).
\endaligned
\end{equation}
We estimate $\|\phi_n^{1/2}u_s\|_2^2$ for any $s \le t$ by first letting $\lambda=M / 2$ in (\ref{proof4th1}),
$$\frac{d}{ds}\|\phi_n^{1/2} u_s
    \|_2^2\le \frac{M  {\alpha^2 } (e^{2\alpha a}+1)}{2}
    \|\phi_n^{1/2} u_s
\|_2^2,$$
and then, by applying the Gronwall inequality:
\begin{equation*}
\|\phi_n^{1/2} u_s \|_2^2\le \exp \left(\frac{M \alpha^2  (e^{2\alpha a}+1)s }{2}\right)
\|\phi_n^{1/2} f\|_2^2.
\end{equation*}
Substituting this into \eqref{proof5th1}, we have
$$\aligned &\bigg(2-\frac{M }
    {\lambda}\bigg)\int_0^t\int \phi_n\,
    \Ga [u_s]\,dm\,ds\\
&\le \|\phi_n^{1/2}f\|_2^2+\frac{2\lambda}{M}
\bigg[
\exp \left( {M \alpha^2 ( e^{2\alpha a} +1)  t}/{2} \right) - 1
\bigg]
\|\phi_n^{1/2}f\|_2^2.\endaligned$$
Setting $\lambda=M $, this becomes
$$\int_0^t\int \phi_n\,\Ga [u_s]\,dm\,ds
\le 2 \exp \left( M \alpha^2 (e^{2\alpha a}+1) t / 2 \right)
\|\phi_n^{1/2}f\|_2^2.$$
The required assertion (\ref{proof6th1}) follows by letting $n\rightarrow\infty$.
\end{proof}

We are in a position to prove Theorem \ref{theorem1}.

\begin{proof}[Proof of Theorem \ref{theorem1}]
 We adopt the notations in the proof of Lemma \ref{lemma2}.
Define a cut-off function $g_n$ for any $n \ge 1$ as follows
\[
g_n (x) := \left( \left( n - a^{-1}d(x,x_0)  \right) \wedge 1 \right)_+.
\]
By Lemma \ref{lemma-1}, $g_n$ belongs to $\dom$.
To the end of the proof,  we show that there exists a sequence $(n_k)_{k\ge0}$ such
that $n_k \to \infty$ as $k\to\infty$, and for every $t>0$,
$$\int_0^t\langle\dot{u}_s, g_{n_k}\rangle \,ds
\rightarrow0\quad\textrm{ as $k\rightarrow\infty$}.$$
Indeed, we can deduce from this and the dominated convergence theorem that
\[
\aligned
\langle T_t f,1\rangle
=\lim_{k \to \infty} \langle u_t,g_{n_k}\rangle
=\lim_{k \to \infty}\langle f,g_{n_k} \rangle
    + \lim_{k \to \infty} \int^t_0 \langle \dot{u}_s, g_{n_k}  \rangle \,ds
=\langle f,1\rangle,
\endaligned
\]
which immediately implies the conservation property.

Since $(u_s)_{s>0}$ solves the heat equation and $g_n \in \dom$,
\begin{equation} \label{eq;311}
\int^t_0 \langle \dot{u}_s, g_n \rangle \,ds =-
\int^t_0 \cE (u_s,g_n) \, ds = -\int^t_0 \big( \cEc (u_s,g_n) + \cE^{(j)} (u_s,g_n)\big)\, ds.
\end{equation}
First, we estimate the second term, the harder one, on the right side.
For any $t>0$,
\begin{equation}\label{proof1th1}
\aligned
\bigg|\int_0^t\cEj (u_s,g_n)\,ds\bigg|&\le \int^t_0 \bigg| \int \Ga^{(j)}(u_s,g_n)\, dm \bigg|\, ds \\
&\le \int_0^t \bigg[ \int\sqrt{\Gamma^{(j)}[u_s]}\,
    \sqrt{\Gamma^{(j)} [g_n]}\, dm\bigg]\, ds\\
&=\int_0^t \bigg[ \int \sqrt{\phi \Gamma^{(j)}[u_s]}
    \sqrt{\phi^{-1} \Gamma^{(j)} [g_n]}\, dm \bigg]\, ds\\
&\le \int_0^t \sqrt{\int \phi \,\Gamma^{(j)}[u_s]\, dm}
    \sqrt{\int \phi^{-1}\,\Gamma^{(j)}[g_n]\,dm}\,ds\\
&\le \sqrt{\int^t_0  \int \phi \,\Gamma^{(j)}[u_s]\, dm \, ds}
\sqrt{\int^t_0  \int \phi^{-1}\,\Gamma^{(j)} [g_n]\,dm \, ds} \\
&= \sqrt{\int^t_0  \int \phi\,\Gamma^{(j)} [u_s]\, dm\, ds}
 \sqrt{ t \int \phi^{-1}\,\Gamma^{(j)} [g_n]\, dm},\endaligned
    \end{equation} where all the
inequalities above follow from the Cauchy-Schwarz inequality.
For any $n>0$, let $A_n$ denote the following annulus associated with the constant $a$
\[A_n=A_n(a)= \overline{B}(x_0, (n+1)a) \setminus B(x_0, (n-2)a).\]
Since supp$(g_n) \subset B(x_0,na)$ and supp$(j(x,dy)) \subset B(x,a)$ for all $x\in X$,
it holds that if $x\notin A_n$,
\[\Gamma^{(j)} [g_n] (x) = \int (g_n(x)-g_n(y))^2\, j(x,dy) =0; \]
if $x\in A_n$,
\[
\aligned\Gamma^{(j)} [g_n](x)\le&  a^{-2} \int d(x,y)^2\,j(x,dy)\\
\le&  a^{-2} \int \big(d(x,y)\wedge a\big)^2\,j(x,dy)\\
\le&  a^{-2} M_j,\endaligned\]
where in the last inequality we have used the fact that $0<a<1$.
Choosing $n$ large enough so that $n\ge a^{-1}\big[4a+2d\,(x_0,\,\textrm{supp} (f)
)\big]$, we get from \eqref{proof111} that
\begin{equation*}
\aligned
\int \phi ^{-1}\,
\Gamma^{(j)} [ g_n ]\, dm &=\int_{A_n} \phi^{-1}\,\Gamma^{(j)} [g_n]\, dm\\
&\le  a^{-2}  M_j  e^{-{a\alpha n}/{2}}\,m(A_n).
\endaligned\end{equation*}
Therefore, by \eqref{proof1th1},
\begin{equation*}
\bigg| \int^t_0 \cEj (u_s,g_n)\, ds \bigg|^2
\le  a^{-2}
{t  M_j  e^{-{a\alpha n}/{2}}\,m(A_n)}
{ \int^t_0 \int \phi \, \Ga^{(j)} [u_s]\, dm\,ds }.
\end{equation*}
In a similar way, we can prove that
\[
\bigg| \int^t_0 \cEc (u_s,g_n)\, ds \bigg|^2
\le
 a^{-2} {t  M_c  e^{-{a\alpha n}/{2}}\,m(A_n)}
{ \int^t_0 \int \phi \, \Ga^{(c)} [u_s]\, dm\,ds }.
\]
Therefore,
\begin{equation}\label{proof2th1}
\bigg|  \int^t_0 \cE (u_s,g_n) \, ds\bigg|^2
\le
2  a^{-2}
{t   M e^{-{a\alpha n}/{2}}\,m(A_n)}
\int^t_0 \int \phi \Ga[u_s]\, dm \, ds.
\end{equation}
We now apply \eqref{proof2th1} and Lemma \ref{lemma2} for \eqref{eq;311} to get that
\begin{equation}\label{proof222}
\aligned &\bigg|\int_0^t \langle\dot{u}_s, g_n
    \rangle \,ds\bigg|^2 \\
& \le  2 a^{-2} {t   M e^{-{a\alpha n}/{2}}\,m(A_n)}
    \int^t_0 \int \phi \,\Ga[u_s]\, dm \, ds \\
& \le   4 a^{-2} tM \|\phi^{1/2}f \|^2 _2
    \exp\left( \frac{M  \alpha^2\,  (e^{2\alpha a}+1)t}{2}
    -\frac{\alpha an}{2}+\log m(A_n)\right).
\endaligned\end{equation}
Finally, we estimate \eqref{proof222} by applying the volume assumption  \eqref{eq;theorem1.1}.
Indeed, according to \eqref{eq;theorem1.1}, there exists a sequence
$(n_k)_{k\ge1}$ such that $n_k \to \infty$ as $k\to\infty$,
and for a  large enough $k\ge 1$,
$$
\aligned \log m(A_{n_k})\le & \,\log V(x_0,(n_k+1)a)\\
\le& \,\big(c_3-1/2\big)\left((n_k+1)a\right)\log
\left((n_k+1)a\right)\\
\le&\,{a\,c_3\,n_k\log n_k},\endaligned
$$
where
$$
c_3=\liminf_{r\rightarrow\infty}
    \frac{\log V(x_0, r)}{r\log r}+1.
$$
Taking $\alpha=4c_3\log n_k$ and $k$ large enough such that
$n_k\ge a^{-1}\big[4a+2d\,(x_0,\,\textrm{supp} (f)
)\big]$, we estimate the right side of \eqref{proof222} to get
\[
\aligned \bigg|\int_0^t \langle\dot{u}_s, g_{n_k}\rangle \,ds\bigg|^2
&\le 4  a^{-2}  t M \|\phi^{1/2}f\|^2_2\,\times\\
&\qquad \times \exp\left(
\frac{M \alpha^2(e^{2\alpha a}+1)t}{2}
    -2ac_3n_k\log n_k+ac_3n_k\log n_k\right) \\
&= 4  a^{-2} tM \|\phi^{1/2}f\|^2_2  \exp\left(
\frac{M \alpha^2 (e^{2\alpha a}+1)t}{2}
    -ac_3n_k\log n_k\right).
\endaligned
\]
Since $e^{2\alpha a}=n_k^{8ac_3}$ and
$8ac_3<1$, the inequality above implies that
for any $t>0$
$$
\lim_{k\rightarrow\infty} \int_0^t \langle\dot{u}_s, g_{n_k}\rangle \,ds =0.
$$
This completes the proof.
\end{proof}
\section{Proof of Theorem \ref{theorem2}: the Recurrence}
\label{proofoftheorem2}
This section is devoted to the proof of the recurrence test, Theorem \ref{theorem2}.
\begin{proof}[Proof of Theorem \ref{theorem2}]
Let $x_0 \in X$ be the reference point in Theorem \ref{theorem2}.
 For $R>2$, set
\[
\theta_{R}(x)=
\left( \left(\frac{R-d(x,x_0)}{R-1}\right)\wedge1\right)_+.
\]
Since $\theta_R$ belongs to $\dom_{\rm loc} \cap L^\infty$
and has compact support, by Lemma \ref{lemma-1}, $\theta_R$ belongs to $\dom$.
According to the condition {(M)} and the chain-rule for a strongly-local Dirichlet form,
$$
\aligned
\cE^{(c)} [\theta_R]&=
\int_X \Gamma^{(c)} [\theta_R]\,dm=\left( \frac{1}{R-1} \right)^2 \int_{\overline{B}(x_0,R)} \Gamma^{(c)} [d]\,dm\\
&\le M_c \left( \frac{1}{R-1} \right)^2 V^{(c)}(x_0,R) \\
&\le \frac{4M_c V^{(c)}(x_0,R)}{R^2}.\endaligned$$
On the other hand, we find that for any $c_1>2$
\begin{align*}
\form^{(j)} [\theta_{R}]&=\iint\left(
\theta_R(x)-\theta_R(y)\right)
^2\,j(x,dy)\,m(dx)\\
&\le\frac{2}{(R-1)^2}\int_{B(x_0,R)}\int_{B(x_0,c_1R)}d(x,y)^2\,j(x,dy)\,m(dx)\\
&\quad+ 2\int_{B(x_0,R)}\int_{B(x_0,c_1R)^c}\, j(x,dy)\,m(dx)\\
&\le \frac{2}{(R-1)^2}\int_{B(x_0,R)}\int_{d(x,y)\le 2c_1R}d(x,y)^2\,j(x,dy)\,m(dx)\\
&\quad+ 2\int_{B(x_0,R)}\int_{d(x,y)\ge(c_1-1)R}\,j(x,dy)\,m(dx),
\end{align*}
where we used the facts that $d(x,y)\le R + c_1 R \le 2c_1 R$
if $x \in B(x_0,R)$ and $y \in B(x_0,c_1R)$; $d(x,y) \ge c_1R -R \ge R_1$ if
$x \in B(x_0,R)$ and $y \notin  B(x_0,c_1R)$.
The last expression is bounded from above by
\begin{align*}
&\le \frac{8\,c_1^2}{(R-1)^2}\int_{B(x_0,R)}
    \int\left(d(x,y)\wedge R\right)^2\,j(x,dy)\,m(dx)\\
&\quad+ \frac{2}{R^2}\int_{B(x_0,R)}\int\left(d(x,y)\wedge
     R\right)^2\,j(x,dy)\,m(dx)\\
&\le \frac{33\,c_1^2}{R^2}\int_{B(x_0,R)}\int\left(d(x,y)\wedge
     R\right)^2\,j(x,dy)\,m(dx).
\end{align*}
Therefore, under the assumption {(M)}, we have that for $c_2=4M_c + 33c_1^2$
\begin{align*}  \form [\theta_R]
    &\le  \frac{1}{R^2}\bigg[4M_c V^{(c)}(x_0,R)+33c_1^2 V^{(j)}(x_0,R)
    \sup_{x\in X^{(j)}} \int\left(
     d(x,y)\wedge     R\right)^2\,j(x,dy)\bigg]\\
&\le \frac{c_2}{R^2}\bigg[ V^{(c)}(x_0,R) +  V^{(j)}(x_0,R)
    \sup_{x\in{ X^{(j)}}} \int\left(
     d(x,y)\wedge     R\right)^2\,j(x,dy)\bigg].
\end{align*}
According to the volume condition \eqref{thm21},
there exists a sequence $(n_k)_{k\ge0}$ such that $n_k \to \infty$
as $k\to\infty$, and
\[
\liminf_{k \to \infty} \cE[\theta_{R_{n_k}}] < \infty.
\]
Applying \cite[Theorem 1.6.3]{FOT94} and \cite[(1.6.1) and (1.6.1')]{SSS},
this completes the proof.
\end{proof}


\section{Examples}\label{examples}
In this section we present some examples to illustrate the power of Theorems \ref{theorem1} and \ref{theorem2}.
Throughout the section, we denote the space of real-valued Lipschitz continuous functions with compact support on a metric space $X$ by $C_0^{\rm Lip}(X)$. For a measure space $(X,m)$ and a quadratic form $\cE$ defined in $L^2(X;m)$, we denote
\[\cE_1 [u]=\|u\|_{L^2}^2 +\cE[u],\]
whenever the right side makes sense.
We start with the following remark for the volume test in Theorem \ref{theorem1}.
\begin{remark}
Let $(X,d,m)$ be a complete metric measure space such that $m$ is a Radon measure with full
support. Assume that there
is a point $x_0 \in X$ such that
\[ \sup_{r>0} \frac{V(x_0,2r)}{ V(x_0,r)} < \infty,\]
where $V(x_0,r)$ denotes the volume of the closed ball centered at $x_0$
with radius $r>0$. This assumption is called the
\emph{volume doubling condition} at point $x_0$,
and it implies that there is a constant $\kappa>0$ such that
\[ \sup_{r>0} \frac{V(x_0,r)}{r^\kappa} < \infty.\]
In particular, condition \eqref{eq;theorem1.1} in Theorem \ref{theorem1} is satisfied.
A typical example which fulfills the volume doubling
condition is a Riemannian manifold with non-negative Ricci curvature.\end{remark}

\subsection{Sharpness Examples}
In the following example, we consider two classes of symmetric jump processes on the so called \emph{$\kappa$-set}.
\begin{example} \label{th23}
Let $(X, |\cdot|,m)$ be a closed $\kappa$-set in $\R^n$ with $0<\kappa\le n$, i.e.,
$|\cdot|$ is the Euclidean distance, and for all $x\in X$ and $r>0$,
\[ m(B(x,r)) \asymp r^\kappa.\]
Here, the symbol $\asymp$ means that the ratio of the left and the right hand sides
is pinched by two positive constants.
Assume that the jump kernel
$j(x,dy)$ has a density $j(x,y)$ with respect to the measure $m(dy)$
such that one of the following two conditions is satisfied with
a constant $\alpha \in (0,2)$:

\medskip

\begin{itemize}
\item[(i)]\ \ $\displaystyle j(x,y)\asymp \frac{1}{|x-y|^{\kappa+\alpha}}
    \I_{\{|x-y|\le 1\}}+\frac{1}{|x-y|^{\kappa+\beta}}
    \I_{\{|x-y|> 1\}},$ where $0<\beta<\infty;$

\medskip

\item[(ii)]\ \ $\displaystyle j(x,y)\asymp \frac{1}{|x-y|^{\kappa+\alpha}}\I_{\{|x-y|\le 1\}}
    +\frac{e^{-c|x-y|}}{|x-y|^{\kappa+\alpha}}\I_{\{|x-y|>1\}},$
where $c>0.$
\end{itemize}
For $u,v\in C_0^{\lip}(X)$,  define
$$
\form(u,v)=\iint_{x\neq y}
{ (u(x)-u(y))(v(x)-v(y))}j(x,y)\,m(dx)\,m(dy).
    $$
Let $\cF$ be the closure of ${C_0^{\lip}(X)}$
with respect to the $\sqrt{\cE_1}$-norm.
\ The symmetric form $( \cE, \cF)$
is a regular Dirichlet form in $L^2(X,m)$, see, e.g., \cite{U4}.
According to Theorems \ref{theorem1} and \ref{theorem2}, the Dirichlet form $(\form,\dom)$ is conservative, and it is
recurrent if additionally $0<\kappa\le\beta\wedge2$ and $0<\kappa\le2$
for the cases (i) and (ii), respectively.
\end{example}
\begin{remark}
Example \ref{th23} is motivated by recent developments
for layered stable processes \cite{HO} and tempering stable processes \cite{RO}. In particular, in case (i) if $\beta=\alpha$, then the associated Hunt process is called a \emph{stable-like process} \cite{CT}.
\end{remark}
\subsection{Disconnected Space}
The following example shows that the state space may be topologically disconnected, and the particles jump between different connected components and it behaves as a jump-diffusion inside a connected component.
\begin{example}
Let $X=\cup_{i \in \Z} X_i$, where for each $i \in \Z$, $X_i=\{
(x_i,i) \in \R^{n+1}: x_i \in {\R}^n \}$. Any point $x$ in $X$ can
be expressed uniquely as $x=(x_i,i)$ with $x_i \in \R^n$ and $i \in
\Z$, and we denote the associated projections by $p:X \to \R^n$ and
$q:X \to \Z$. For any $x,y\in X$, the distance $d$ is given by
\[
d(x,y) = |p(x)-p(y)| +|q(x)-q(y)|,
\]
where $|\cdot|$ is the Euclidean distance. Let $m(dx)=\sum_{i\in \Z}
m_i(dx_i)$ be a measure on $X$ such that for each $i\ge1$,
$m_i(dx_i) = \Psi(x_i)\,dx_i$ is a measure on $X_i$, where $\Psi \in
C(\R^n)$ is a positive function, and $dx_i$ is the $n$-dimensional
Lebesgue measure. Clearly, $m$ is a Radon measure on $X$. The state
space is the triple $(X,d,m)$.

For any $ u\in C^{\rm Lip}_0 (X)$, define
$$\cE[u]=\cE^{(c)}[u]+\cE^{(j)}[u],$$ where
\[
\cE^{(c)} [u] = \int_X |\nabla u|^2\, dm,
\]
\[
\cE^{(j)} [u] = \int_X \int_{x \neq y} \left( u(x)-u(y) \right)^2 j(x,y)\, m(dx) m(dy),
\]
and
\[ j(x,y) \asymp
\frac{d(x,y)^{-(n+\alpha)}\I_{\{d(x,y) < 1\}}
+d(x,y)^{-(n+\beta+1)} \I_{\{d(x,y) \ge 1\}}
}{ \Psi (p(x)) + \Psi (p(y))}, \quad x,y \in X
   \]
with some constants $0<\al<2$ and $\be>0$. Let $\cF$ be the closure
of ${C_0^{\lip}(X)}$ with respect to the $\sqrt{\cE_1}$-norm. Since
for any $x\in X$\begin{align*}
&\int_{x \neq y} \left( 1 \wedge d(x,y)^2 \right)j(x,y)\, m(dy) \\
\le&
\int_{0<d(x,y)<1}\frac{ d(x,y)^{-(n+\al-2)} \Psi(p(y))\, dp(y)}{\Psi(p(x)) + \Psi(p(y))}+
\int_{d(x,y) \ge 1}\frac{ d(x,y)^{-(n+\be+1)} \Psi(p(y))\, dp(y)}{\Psi(p(x)) + \Psi(p(y))} \\
\le&
\int_{0<d(x,y)<1} d(x,y)^{-(n+\al-2)} \, dp(y)+
2 \sum_{k \ge 0} \int_{|p(x)-p(y)| \ge k+1} { |p(x)-p(y)|^{-(n+\be+1)}\, dp(y)},
\end{align*}
which is bounded from above by some absolute constant $c>0$,
it follows form the proof of \cite{U4} that $(\cE, \cF)$ is a regular Dirichlet form
in $L^2(X,m)$.

According to the arguments above, we can easily claim that the
condition (M) is satisfied. Therefore, by Theorem \ref{theorem1}, if
there is a constant $c>0$ such that for $r>0$ large enough
\begin{equation} \label{ex;eq1}
\sum_{0 \le k \le [r]} \int_{B(0,{[r]-k})} \Psi(z)\,dz \le r^{cr},
\end{equation}
where $dz$ is the $n$-dimensional Euclidean measure and $[r]$ is the
least integer such that $[r] \ge r$, then the Dirichlet form
$(\form,\dom)$ is conservative. For instance, \eqref{ex;eq1} is
satisfied, if $\Psi (x) \le |x|^{|x|} \ln |x|$ for $|x|$ large
enough.

For the recurrence, we additionally assume that there are two constants
$c_0$, $c_1>0$ such that

\begin{equation} \label{eq; 5.16}
j(x,y) \le \frac{\I_{\{d(x,y) \le c_0\}}}{ d(x,y)^{1+\al} }
\end{equation}
 and \begin{equation} \label{eq; 5.15}
\Psi (x) \le c_1 |x|^{1-n} \qquad \mbox{for $|x|$ large enough}.
\end{equation}
Condition \eqref{eq; 5.15} will imply that for any point $x_0 \in X$,
\[ \liminf_{r\to\infty} \frac{V(x_0, r)}{r^2} \le 2  \liminf_{r\to\infty} \frac{1}{r^2}
    \sum_{0 \le k \le [r]} \int_{B(x_0,[r]-k)} \Psi (x)\, dx < \infty.   \]
Next, by \eqref{eq; 5.16}, there is a constant $c_2>0$ depending
only on the dimension such that
\begin{align*}
\omega (r) & \le \sup_{x \in X} \int_{X} d(x,y)^2 j(x,y) \Psi(p(y))\, dp(y) \\
&\le c_1 \sup_{x \in X} \int_{d(x,y) \le c_0}{d(x,y)^{1-\al}} |p(y)|^{1-n}\, dp(y) \\
&\le 2c_1 c_2\sum_{0 \le k \le [c_0]} \int_0^{[c_0]-k}
{r^{1-\al}}{dr}<\infty.
\end{align*}
Therefore, $(\cE,\cF)$ is recurrent by Theorem \ref{theorem2}.
\end{example}

\subsection{Volume Tests}
The first volume test for non-local Dirichlet forms to be
conservative was obtained in \cite[Main Result]{MasamuneUemura.11},
and then refined in \cite[Theorem 1.1]{Grigor'yanHuangMasamune.10}.
It is easy to construct an example, which is not covered by these
tests but by Theorem \ref{theorem1}. Here, we illustrate this by
using a weighted Euclidean space as well as a model manifold.

\begin{example}\label{exp2}
Let $(\R, |\cdot|, m)$ be a weighted Euclidean space, where $|\cdot|$ is the
Euclidean distance and the measure is $m(dx)=e^{2\lambda |x|}\,dx$ for
some $\lambda>0$.
For $u\in C_0^{\lip}(\R)$,  define
$$
\form[u]=\iint_{x\neq y}\left( u(x)-u(y) \right)^2 j(x,y)\,m(dx)\,m(dy),
    $$ where $$
j(x,y)=\Big(e^{-\lambda(|x|+|y|)}\Big)\I_{\{|x-y|\le 1\}}.
$$
Let $\cF$ be the closure of ${C_0^{\lip}(\R)}$ with respect to the
$\sqrt{\cE_1}$-norm. The symmetric form $( \cE, \cF)$ becomes a
regular Dirichlet form in $L^2(\R,m)$, see, e.g., \cite{U4}. Let
$j(x,dy)=j(x,y)\,m(dy)$. It holds that
$$
\aligned \sup_{x\in \R}\int\big(1\wedge |x-y|^2\big)\,j(x,dy)
    &=\sup_{x\in \R}\int_{\{|y-x|\le 1\}}|x-y|^2\,j(x,y)\,m(dy)\\
&=\sup_{x\in\R}e^{-\lambda|x|}
\int_{\{|z|\le 1\}}z^2e^{\lambda|x-z|}\,dz\\
&\le\int_{\{|z|\le 1\}}z^2e^{\lambda|z|}\,dz<\infty. \endaligned
$$ On the other hand, it is easy to see that in this example
\eqref{eq;theorem1.1} is also satisfied. Therefore, according to
Theorem \ref{theorem1}, the Dirichlet form $ ( \cE, \cF) $ is
conservative.

However, since $x\mapsto e^{-r|x|}\notin L^1(\R, m)$ for any
$r\le 2\lambda$, this example is
not covered by \cite[Main Result]{MasamuneUemura.11}.
\end{example}

\begin{example}[Model manifolds. See, e.g., \cite{Grigor'yan.99}]
Let $(\mathds S^n, g)$ be the $n$-dimensional  unit sphere with $n
\ge 1$. A \emph{model manifold} $M = (0, +\infty) \times \mathds
S^n$ is a Riemannian manifold with Riemannian tensor
 \[ dr^2 +
\sigma^2 (r) g \]where $\sigma$ is a locally-Lipschitz continuous
positive function on $[0,+\infty)$ such that $\sigma(0)=0$ and
$\sigma'(+0)=0$. Thanks to these two conditions, the manifold $M$ is
geodesically complete, and so it satisfies the assumption for the
state space as explained in Introduction. Let $dm = \omega_n
\sigma^n (r)\, dr$ be a measure on $M$, where $\omega_n$ is the
volume of $\mathds S^n$.

For any $ u\in C^{\rm Lip}_0 (M)$, define
$$\cE[u]=\cE^{(c)}[u]+\cE^{(j)}[u],$$
where
\[
\cE^{(c)} [u] = \int_M |\nabla u|^2\, dm,
\]
\[ \cE^{(j)}[u] := \iint_{M \times M \setminus {\rm diag}} \left( u(x) - u(y)
\right)^2\,j(x,y)\,m(dy)\,m(dx)\]
 and
\[j(x,y) = \Bigg[ \frac{\I_{\{d(x,y)<1\}}}{\sigma (r(x)) \sigma (r(y))} \Bigg]^n.\]
Let $\cF$ be the closure of ${C_0^{\lip}(M)}$ with respect to the
$\sqrt{\cE_1}$-norm. It is easy to check that the symmetric form $(
\cE, \cF)$ is a regular Dirichlet form in $L^2(M,m)$.

By \cite{G.54}, it is known that (M-2) is satisfied. On the other
hand, since $$\sup_{x,y \in M} j(x,y) \sigma^n(r(y)) \le 1,$$ we
obtain that\begin{align*}
M_j &= \sup_{x \in M} \int_M (1 \wedge d(x,y)^2)\, j(x,dy) \\
& \le \sup_{x \in M} \int_M d(x,y)^2\, j(x,y) m(dy)  \\
& \le \sup_{x \in M} \int_{ d(x,y)<1 } d(x,y)^2\, \omega_ndy \\
&\le \omega_n.
\end{align*}
Therefore, (M-3) is also satisfied. Since (M-1) clearly follows, we
can apply our main theorem. For example, if $\sigma$ satisfies
\[ \sigma(r) \asymp [ r^r (1 + \ln r) \vee 1]^{1/n}. \]
then for any fixed $x_0 \in M$,
\[ r^{r/2} < V(x_0, r) < 2 r^r \quad \mbox{for large $r>0$}.\]
Therefore, $(\cE,\cF)$ is conservative by Theorem \ref{theorem1}. We
note that this model manifold $M$ does not satisfy the volume tests
in \cite{MasamuneUemura.11, Grigor'yanHuangMasamune.10}.
\end{example}

\subsection{A mixed-type Laplacian on Graphs}
A graph admits natural different ``Laplacians\rq\rq{}; namely, a
physical Laplacian, a combinatorial Laplacian, and a quantum
Laplacian. The former two are non-local operators, and the last one
is a local operator. The combinatorial Laplacian is bounded, and
so the corresponding process always is conservative. The
conservativeness of the process associated with the physical
Laplacian was studied in \cite{Dodziuk.06, DodziukMathai.06,
Weber.10, Wojciechowski.09, Grigor'yanHuangMasamune.10}. The
conservativeness and recurrence of the process generated by the
quantum Laplacian was studied in \cite{Sturm.94}. In the following
example, we consider the sum of a physical Laplacian and a quantum
Laplacian, and study its conservativeness.

\medskip

Let $X=(V,E)$ be a locally finite graph, where $V$ and $E$ are the
sets of vertices and edges, respectively. Let $\mu$ be a positive
function on $X$, and $\omega: X \times X \to [0,\infty)$ be a
symmetric nonnegative function, such that $\omega(x,y)=0$ whenever
$x=y$ for $x,y \in X$ or at least one of $x$ and $y$ does not belong
to $V$. Now, we recall the \emph{standard adapted distance} $d$ in
\cite{Grigor'yanHuangMasamune.10}. For any $x$, $y\in X$, $x \sim y$
means that $x$, $y$ are neighbors; that is, $(x,y)\in E$.  For
all $x,y \in V$ with $x \sim y$, define
\[
\sigma(x,y) = \min \Bigg\{  \frac{1}{\sqrt{ \mbox{deg} (x) }}, \frac{1}{\sqrt{ \mbox{deg} (y) }},  1 \Bigg\},
\]
where
\[\mbox{deg}(x)=\frac{1}{\mu (x)} \sum_{y: y \sim x} \omega (x,y). \]
 It naturally induces a metric $d$ on $V$ as
\[
d(x,y)=\inf\bigg\{ \sum_{i=0}^{n-1} \sigma(x_i,x_i+1): \mbox{$x_0,
\cdots, x_n$ is a chain connecting $x$ and $y$} \bigg\}.
\]
The metric $d$ can be extended to $X$ by linear interpolation.
We assume that the lengths of all edges $e \in E$ are
uniformly bounded from below by a positive constant.
This implies that $(X,d)$ is a metrically complete space; in particular,
our assumption on the space is satisfied.

We further assume that each edge $e \in E$
is isometric to an interval of $\R$, which yields the measure $dx$ on $e$.
The space $(X,d)$ is a \emph{metric graph}.
Consider the following measure $m$ on $X$:
\[ m := \delta_{E}\phi\, dx +  \delta_{V} \mu, \] where $\phi$ is a continuous positive function on $E$.

For $u \in C^{\rm Lip}_0(X)$, define
\[ \form [u] := \form^{(c)} [u] + \form^{(j)} [u] , \]
where \[
\cE^{(c)} [u] = \int_{E} \left( \frac{\partial u}{\partial x} \right)^2\,dm,
\]
and \[
\cE^{(j)} [u] = \sum_{x,y \in V} \left( u(x) - u(y) \right)^2\, \omega (x,y).
\]
The generators associated with $\form^{(c)}$
and $\form^{(j)}$ are called the \emph{quantum graph}, see, e.g.\ \cite{Kuchment.04}
and the physical Laplacian, respectively.
Let $\cF$ be the closure of $C^{\rm Lip}_0(X)$ with respect to the $\sqrt{\cE_1}$-norm.
We have
\begin{lemma} The form $(\cE,\cF)$ is a regular Dirichlet form. \end{lemma}

\begin{proof}
First, we claim that $C^{\rm Lip}_0(X)$ is dense in $L^2(X;m)$.
Let $x_0$ be a fixed point in $V$. For any $u \in L^2(X;m)$ and any $\ep>0$,
choose $R>0$ so large that
there is a function $v_\ep \in C^\infty_0 ( B(R) \cap E)$ which satisfies
\[ \| v_\ep - u|_E \|_{L^2(E;dx)} < \ep ,\]
and that the function $w_\ep  = \mathds{1}_{B(R)} u$ satisfies that
\[ \| w_\ep - u|_V \|_{L^2(V;\mu)} < \ep, \] where $B(R):=B(x_0,R)$.
Set $\tilde u_\ep = \delta_E v_\ep + \delta_V w_\ep$.
For any $x \in B(R)$ and $e \in E$ with $x \sim e$ (i.e., $x  \in e$),
let $\delta=\delta(x,e)$
be a positive number such that $\delta < |e| / 2$, and modify $\tilde u_\ep$ on $e \cap B(x,\delta)$ so that $\tilde u_\ep$
is linear and continuous on $e \cap B(x,\delta)$.
Furthermore, since $B(R) \cap V$ is finite, by the Hopf-Rinow type property of locally finite graphs
\cite{HKMW},
we are able to do this modification for any $x \in B(R) \cap V$
and any $e \in E$ with $x \sim e$. Consequently,
we obtain a sequence of functions
$u^\delta_\ep \in C^{\rm Lip}_0(B(R))$
which converges to $u$ in $L^2(X;m)$ as $\delta$, $\ep \to 0$.
The required claim is proved.

Next, we verify that $(\form,C^{\rm Lip}_0(X))$ is closable.
Let $(u_n)_{n\ge1}\subset C^{\rm Lip}_0(X)$ be an $\form_1$-Cauchy sequence
such that $u_n \to 0$ in $L^2(X;m)$ as $n\to\infty$.
One can easily prove that $\cE^{(c)} [u_n|_E] \to 0$ as $n\to\infty$,
since $\cE^{(c)}$ is equivalent to the Dirichlet
integral of an open interval.
Moreover, if $v \in C^{\rm Lip}_0(X)$, then
\[ \cE^{(j)} (u_n|_V, v|_V) = \sum_{x,y \in V}
     \left( u_n (x) - u_n (y)\right) \left( v(x) - v(y)\right) \omega(x,y)  \to 0
     \quad \mbox{as $n\to\infty$}.\]
Therefore, the desired claim follows and we denote
the closure of $(\form,C^{\rm Lip}_0(X))$ by $(\cE,\cF)$.

The Markov property of $(\cE,\cF)$ follows immediately
from the definition of $\cE$. Finally, since $C_0 \cap \cF$ is both dense in $C_0$
and $\cF$ with respect to the $\sup$-norm and the $\cE_1$-norm, respectively,
$(\cE,\cF)$ is regular.
\end{proof}
It is easy to see that the conditions (M-1) and (M-2)
are satisfied since $X^{(c)}=E$.
Moreover, since $\cE^{(j)}$ can be expressed as
\[
\cE^{(j)} [u] = \iint_{X \times X} \left( u(x) - u(y)\right)^2
     \frac{\omega (x,y)}{\mu (x) \mu(y)} \, m (dy)\, m(dx),
\]
the associated jump kernel $j$ and $\Gamma_j$ have the forms
\[ j(x,dy) = \frac{\omega (x,y)}{\mu (x) \mu(y)} m(dy) \]
and
\[ \Gamma_j [u](x) =  \int_X \left( u(x) - u(y)\right)^2
    \frac{\omega (x,y)}{\mu (x) \mu(y) } m(dy) \quad \mbox{for any $x \in X$.}\]
Clearly, (M-3) is satisfied. Therefore the Dirichlet form $(\cE,\cF)$ satisfies the condition (M).

To state our main result in this subsection, we need some notations.
Denote by $\rho$ the graph distance extended to $X$, and by $B_\rho (x_0,R)$
 the associated ball at $x_0 \in V$ with radius $R>0$.
For any $n\in \N$, let $S_\rho(x_0,n)$ be the \lq\lq{}boundary\rq\rq{} $B_\rho (x_0,n) \setminus B_\rho (x_0,n-1)$.
\begin{proposition} \label{proposition;5.8}
If $\mu$ is the counting measure and there are a point $x_0 \in V$ and a constant $C>0$ such that
\begin{equation} \label{eq;5.14}
m (S_\rho (x_0, n)) \le Cn^2\quad \mbox{for all large enough $n \in \N$},
\end{equation}
then $(\cE,\cF)$ is conservative.
\end{proposition}
\begin{proof}
The condition (\ref{eq;5.14}) implies that for any $x \in V$,
\begin{equation} \label{equation;5.21}
d(x_0,x) \ge \delta \log \rho(x_0,x),
\end{equation}
where $\delta>0$ is a constant depending only on $C$ in (\ref{eq;5.14})
(see \cite{Grigor'yanHuangMasamune.10}). Let $\overline{xx\rq{}}$ be the edge with boundary $\{x,x\rq{}\}$. Let $y \in X$ and $x,x\rq{} \in V$ such that $y \in \overline{xx\rq{}}$. Without loss of generality,
we assume that $\rho(x_0,y) \le \rho(x_0,x\rq{})$.
By using \eqref{equation;5.21},
the triangle inequality and the fact that $d(x,x\rq{}) \le \rho(x,x') = 1$, we find that
\[\rho (x_0, y) \le e^{d(x_0,x\rq{})/\delta} \le e^{1/\delta} e^{d(x_0,x)/\delta}.\]
Since $d(x_0, y) \ge d(x_0,x) \wedge d(x_0,x\rq{})$, we obtain that
there is a constant $c>0$ such that
\[ \rho (x_0,y) \le c e^{d(x_0,y)/\delta}\quad \mbox{ for any $y \in X$}.\]
It follows that there exists a constant $b>0$ such that
\[m (B_d (x_0,r)) \le m (B_\rho (x_0, c e^{r/\delta})) \le \exp (br)\quad
\mbox{ for all large enough $r>0$}.\]
Therefore, $(\form,\dom)$ is conservative by Theorem 1.
\end{proof}
\begin{remark}
By an example of R.\ Wojciechowski \cite{Wojciechowski.11}, the
boundary volume growth of quadratic rate \eqref{eq;5.14}
 is sharp.
The second part of Proposition \ref{proposition;5.8} was obtained in
\cite{Grigor'yanHuangMasamune.10} for a physical Laplacian on a
graph.

On the other hand, it is easy to check that the condition
\eqref{eq;5.14} is satisfied, if there is a constant $C>0$ such that
\begin{enumerate}
\item $\mu (S_\rho (x_0,n)) \le C n^2$ for all large enough $n \in \N$
\item $\phi (x)  \le C \rho(x_0,x)^{-2}$ for every $x \in X$.
\end{enumerate}
Indeed, the first condition implies that there are at most
$(C n^2)^2$-many edges in $S_\rho (x_0,n)$ connecting vertices in $S_\rho (n)$ and
$S_\rho (n-1)$. The second condition then implies that there is a constant $c>0$ such that
\[ m ( S_\rho (x_0,n) \cap E) \le \frac{C^3 n^4 }{(n-1)^2} \le cn^2
\quad
\mbox{for all large enough $n$}.\]
This together with the first condition yields \eqref{eq;5.14}.

\end{remark}
\section{Acknowledgements}
Part of this work was done when J.\  Masamune and J.\ Wang visited
TU Dresden as a visitor and a Humboldt fellow,
respectively. They are grateful for Professor Ren\'{e} L.\ Schilling
for providing them with nice working environment and for stimulating
discussions. Financial support through National Natural Science
Foundation of China (No.\ 11126350) for J.\ Wang is gratefully
acknowledged.

\end{document}